\include{kbordermatrix.sty}
\documentclass[12pt]{amsart}
\usepackage{hyperref,amsmath,amssymb,enumerate,kbordermatrix,multirow}
\newtheorem{theorem}{Theorem}[section]
\newtheorem{claim}{}[theorem]
\newtheorem{lemma}[theorem]{Lemma}

\theoremstyle{definition}

\renewcommand{\kbldelim}{(}
\renewcommand{\kbrdelim}{)}

\newcommand{\bF}{\mathbb F}
\newcommand{\bR}{\mathbb R}
\newcommand{\bZ}{\mathbb Z}

\newcommand{\cF}{\mathcal{F}}

\newcommand{\cG}{\mathcal{G}}
\newcommand{\cL}{\mathcal{L}}

\newcommand{\cM}{\mathcal{M}}
\newcommand{\cP}{\mathcal{P}}
\newcommand{\cO}{\mathcal{O}}

\newcommand{\cT}{\mathcal{T}}

\newcommand{\sfourmatrix}[4]{\left(\begin{smallmatrix}#1 & #2 \\ #3 & #4\end{smallmatrix}\right)}
\newcommand{\fourmatrix}[4]{\left(\begin{matrix}#1 & #2 \\ #3 & #4\end{matrix}\right)}
\newcommand{\ol}{\overline}

\newcommand{\olpg}{\ol\PG}
\newcommand{\bpg}{\olpg}
\newcommand{\hpg}{\widehat\PG}
\newcommand{\sqbr}[1]{\left\lbrack #1 \right\rbrack}
\DeclareMathOperator{\si}{si}
\DeclareMathOperator{\aut}{Aut}

\DeclareMathOperator{\cl}{cl}
\DeclareMathOperator{\spn}{span}

\DeclareMathOperator{\rank}{rank}

\DeclareMathOperator{\col}{col}
\DeclareMathOperator{\PG}{PG}
\DeclareMathOperator{\GF}{GF}
\DeclareMathOperator{\AG}{AG}
\newcommand{\cPG}{\mathcal{P}\mathcal{G}}
\DeclareMathOperator{\rowspace}{row}
\newcommand{\ba}{\ol A}
\newcommand{\ha}{\widehat A}
\newcommand{\elem}{\epsilon}
\newcommand{\del}{\setminus \!}
\newcommand{\con}{/}
\newcommand{\dcon}{/\!\!/}
\author{Peter Nelson, Stefan H.M. van Zwam}
\title[$\GF(q)$-regular matroids]{Matroids representable over fields with a common subfield}

\begin{document}
\begin{abstract}
A matroid is \emph{$\GF(q)$-regular} if it is representable over all proper superfields of the field $\GF(q)$. We show that, for highly connected matroids having a large projective geometry over $\GF(q)$ as a minor, the property of $\GF(q)$-regularity is equivalent to representability over both $\GF(q^2)$ and $\GF(q^t)$ for some odd integer $t \ge 3$. We do this by means of an exact structural description of all such matroids. 
\end{abstract}
\thanks{ This research was supported by the
National Science Foundation grant 1161650.}
\maketitle
\section{Introduction}

For a field $\bF_0$, we say a matroid $M$ is \emph{$\bF_0$-regular} if $M$ is representable over every field $\bF$ having $\bF_0$ as a proper subfield. 

Let $n \ge 2$ be an integer, $q$ be a prime power, and $N$ be a $\PG(n-1,q)$-restriction of a matroid $M \cong \PG(n-1,q^2)$. Let $L_0$ be a line of $N$ and $x \in \cl_M(L_0) - L_0$. We denote by $\hpg(n-2,q)$ any matroid isomorphic to $\si((M \con x)|E(N))$. If $n \ge 3$ and $f \in E(N)-L_0$, then we denote by $\bpg(n-1,q)$ any matroid isomorphic to $M|(E(N) \cup \cl_M(\{x,f\}))$. (We will show later that these matroids are uniquely determined up to isomorphism.) A matroid $M$ is \emph{round} if $E(M)$ is not the union of two hyperplanes, or equivalently if $M$ is infinitely vertically connected. Our main theorem is the following:

\begin{theorem}\label{main}
	Let $q$ be a prime power and $M$ be a round rank-$r$ matroid with a $\PG(12q^{12}+19,q)$-minor. The following are equivalent:
	\begin{enumerate}
		\item\label{mainone} $M$ is $\GF(q)$-regular;
		\item\label{maintwo} $M$ is representable over $\GF(q^2)$ and $\GF(q^{t})$ for some odd integer $t \ge 3$; and
		\item\label{mainthree} $\si(M)$ is a restriction of either $\hpg(r-1,q)$ or $\bpg(r-1,q)$. 
	\end{enumerate}
\end{theorem}

This exactly characterises all $\GF(q)$-regular matroids that are sufficiently `rich' and highly connected; the equivalence of (\ref{mainone}) and (\ref{maintwo}) is strongly reminiscent of Tutte's characterisation of regular matroids of the usual sort, and motivates our use of the word. This equivalence may hold for all matroids (this has essentially been conjectured for $q = 2$ in [\ref{pvz}, Conjecture 6.8]), but the characterisation in (\ref{mainthree}) requires some extra hypotheses, and we briefly discuss the ones we chose. 

As one could otherwise construct counterexamples by taking $2$-sums and $3$-sums, some connectivity assumption is needed. However, the hypothesis of roundness is probably overkill. The theorem likely holds for vertically $4$-connected matroids, and many of our techniques apply in this more general setting. Proving a `vertically $4$-connected' version of the theorem would require analysis of how the structure in (\ref{mainthree}) propagates over $4$-separations. 

The hypothesis of having some sort of underlying `richness', here a large projective geometry minor, is also necessary; the structure in (\ref{mainthree}) does not describe all vertically $4$-connected $\GF(q)$-regular matroids. Indeed, Gerards [\ref{gthesis}] defined a class of signed-graphic matroids representable over every field with at least three elements; this class contains counterexamples to our theorem of arbitrarily high branch-width. However, Gerards' counterexamples are nearly planar; it is possible that a very similar structure to that in (\ref{mainthree}) holds for all vertically $4$-connected matroids with a large enough clique minor. Round $\GF(q^2)$-representable matroids of huge rank have a large clique minor [\ref{ggwep}], so in the round setting it is possible that our hypothesis of a large projective geometry minor could be replaced with a `large rank' hypothesis with few other changes to the theorem statement.  

Though the material in this paper is self-contained, sections~\ref{tanglesec1} and~\ref{tanglesec2} make essential use of the theory of tangles and some currently unpublished techniques due to Geelen, Gerards and Whittle [\ref{ggwnonprime}].

\section{Preliminaries}
	
	We largely follow the notation of Oxley [\ref{oxley}].	We also write $\elem(M)$ for $|\si(M)|$. For a positive integer $n$, we denote the set $\{1, \dotsc, n\}$ by $[n]$. Finally, if $\bF_0$ is a subfield of a field $\bF$ and $A$ is an $\bF$ matrix, we write $\rowspace_{\bF_0}(A)$ for the vector space containing all linear combinations of the rows of $A$ with coefficients in $\bF_0$. We define $\col_{\bF_0}(A)$ similarly. 
	
	The versions of connectivity we consider are all `vertical'; for $k \in \bZ^+ \cup \{\infty\}$ a set $A \subseteq E(M)$ is $\emph{vertically $k$-separating}$ in $M$ if $\lambda_M(A) < k$ and $\min(r_M(A), r(M \del A)) \ge k$, and $M$ is $\emph{vertically $k$-connected}$ if $M$ has no vertically $k'$-separating subsets for any $k' \le k$. $M$ is \emph{round} if it is vertically $\infty$-connected; for example cliques, projective geometries and non-binary affine geometries are round. A matroid $M$ is vertically $k$-connected if and only if its simplification is vertically $k$-connected. Moreover if $M$ is vertically $k$-connected then $M \con e$ is vertically $(k-1)$-connected for each $e \in E(M)$; in particular if $M$ is round then so is $M \con e$. We will use the following slight strengthening of a well-known result on connectivity; see [\ref{oxley}, Theorem 8.5.7].
	
	\begin{theorem}[Tutte's Linking Theorem]\label{linking}
		Let $M$ be a matroid and $A,B \subseteq E(M)$ be disjoint sets. There is a minor $N$ of $M$ so that $E(N) = A \cup B$, $N|A = M|A$, $N|B = M|B$ and $\lambda_N(A) = \kappa_M(A,B)$. 
	\end{theorem}
	
	To avoid complications arising from inequivalent representations, we will often consider matroids defined by a representation rather than axiomatically. If $\bF$ is a field, then an \emph{$\bF$-represented matroid} on ground set $E$ is a pair $M = (U,E)$, where $U$ is a subspace of $\bF^E$. This represented matroid has rank function given by $r_M(X) = \dim(U[X])$ for each $X \subseteq E$, where $U[X]$ is the projection of $U$ onto $\bF^{X}$. Where confusion might arise, we refer to a matroid defined in the usual way as an \emph{abstract} matroid; if $M$ is an $\bF$-represented matroid then we write $\tilde{M}$ for the abstract matroid with the same rank function as $M$. 
	
	Given a matrix $A \in \bF^{X \times E}$, we write $M(A)$ for the $\bF$-represented matroid $(\rowspace(A),E)$ and $\tilde{M}(A)$ for the associated abstract matroid; here $A$ is an \emph{$\bF$-representation} of $M(A)$. We also need to formalize deletion and contraction in this context; given an $\bF$-representation $A$ of an $\bF$-represented matroid $M$ and a set $X \subseteq E(M)$, we write $M \del X$ for the $\bF$-represented matroid $M(A[E(M)-X])$. It is easiest to define contraction in terms of duality; if $M = (U,E)$ is an $\bF$-represented matroid then let $M^* = (U^{\perp},E)$, where $U^\perp = \{v \in \bF^E: \langle v,u \rangle = 0 \text{ for all $u \in U$}\}$, and $M \con X = (M^* \del X)^*$. Given a particular representation $A$, this is equivalent to the usual matrix interpretation of contraction where we row-reduce and take a submatrix of $A$. We extend these definitions to define a \emph{minor} and \emph{restriction} of an $\bF$-represented matroid, as well as extending all other usual matroidal notions such as connectivity. 
	
	If $\bF_0$ is a subfield of $\bF$, then two $\bF$-matrices $A_1,A_2$ are \emph{$\bF_0$-row-equivalent} if one can be obtained from the other by elementary row-operations only involving coefficients in $\bF_0$. Furthermore, the matrices $A_1,A_2$ are \emph{$\bF_0$-projectively equivalent} if there is a matrix $A_1'$ that is $\bF_0$-row-equivalent to $A_1$  that can be obtained from $A_2$ by scaling columns by nonzero elements of $\bF_0$. We also say that the $\bF$-represented matroids $M(A_1)$ and $M(A_2)$ are $\bF_0$-projectively equivalent. If $\bF_0 = \bF$ then we just say the matrices or represented matroids are \emph{projectively equivalent}, and write $A_1  \approx A_2$ and $M(A_1) \approx M(A_2)$. It is clear that if $M \approx M'$ then $\tilde{M} = \tilde{M'}$. For each integer $n$, let $\cPG(n-1,q)$ denote the set of $\GF(q)$-matrices $G$ with row-set $[n]$ satisfying $\tilde{M}(G) \cong \PG(n-1,q)$. 
\section{Algebra}
	
 We frequently consider an extension field $\bF$ of a field $\bF_0$; our main theorem applies just when $\bF_0 = \GF(q)$ and $\bF = \GF(q^2)$, but some lemmas apply for arbitrary $\bF_0$. When the extension has degree $2$ with $\bF = \bF_0(\omega)$, we often use the fact that $\bF$ is a dimension-$2$ vector space over $\bF_0$ with basis $\{1,\omega\}$. We require a few lemmas relating $\bF_0$ and $\bF$ in various contexts; the first is proved in [\ref{n13}]. 
\begin{lemma}\label{confinement}
	Let $n \ge 3$ be an integer, $q$ be a prime power, and $\bF$ be a field with a $\GF(q)$-subfield. If $A$ is an $\bF$-matrix with $M(A) \cong \PG(n-1,q)$, then $A$ is projectively equivalent to a $\GF(q)$-matrix. 
\end{lemma}

 We will apply the next lemma in the case where $j = 2$ and $h = 3$. 

\begin{lemma}\label{gfqvectorinplane}
	Let $\bF = \bF_0(\omega)$ be a degree-$2$ extension field of a field $\bF_0$ and let $j,h,t \in \bZ^+$ satisfy $2j > h$ and $j,h \le t$. If $V$ is an $h$-dimensional subspace of $\bF_0^t$ and $U$ is a $j$-dimensional subspace of $\bF^t$ such that $U \subseteq \spn_{\bF}(V)$, then $U \cap V$ is nontrivial. 
\end{lemma}
\begin{proof}
	Let $\{b_1, \dotsc, b_h\}$ be a basis for $V$ and let $W = \spn_{\bF}(V)$, noting that each $w \in W$ is expressible in the form $\sum_{i=1}^h(\lambda_i + \omega \mu_i)b_i$ for some unique $\lambda,\mu \in \bF_0^h$. Let $\varphi: W \to \bF_0^{2h}$ be the invertible linear transformation defined by $\varphi\left(\sum_{i = 1}^h(\lambda_i + \omega \mu_i)b_i\right) = (\lambda_1, \dotsc, \lambda_h, \mu_1, \dotsc, \mu_h)$. Now $\varphi(U)$ and $\varphi(V)$ are subspaces of $\bF_0^{2h}$ with $\dim(\varphi(U)) = 2j$ and $\dim(\varphi(V)) = h$, so $\dim(\varphi(U) \cap \varphi(V)) = 2j + h - 2h > 0$. Therefore $U \cap V$ is nontrivial, as required. 
\end{proof}

\begin{lemma}\label{sunday0}
	Let $\bF_0$ be a field and $\bF = \bF_0(\omega)$ be a degree-$2$ extension field of $\bF_0$. Let $h,d,n \in \bZ_0^+$ satisfy $h \le d$ and let $A,B \in \bF_0^{d \times n}$ be matrices such that $\rank(A + \omega B) = d$. If $\rank\binom{A}{B} = 2d - h$ then there is a rank-$h$ matrix $Q \in \bF^{h \times d}$ such that $Q(A + \omega B)$ is an $\bF_0$-matrix. 
\end{lemma}
\begin{proof}
	Let $\omega^2 = s + \omega t$ for $s,t \in \bF_0$. If $\rank\binom{A}{B} = 2d - h$ then there are matrices $Q_1,Q_2 \in \bF_0^{h \times d}$ such that $(Q_1 | Q_2)\binom{A}{B} = Q_1 A + Q_2 B = 0$ and $\rank(Q_1 | Q_2) = h$. Let $Q = (\omega - t)Q_1 + Q_2$; we have $Q(A + \omega B) = (Q_2A - tQ_1A + sQ_1B) + \omega(Q_1 A + Q_2B)$ which is an $\bF_0$-matrix. 
	
	It remains to show that $\rank(Q) = h$. If not, then there are row vectors $x,y \in \bF_0^h$ such that $x + \omega y \ne 0$ and $(x + \omega y)Q = 0$. This gives $(xQ_2 - txQ_1 + syQ_1) + \omega(xQ_1 + yQ_2) = 0$, implying that 
	\begin{equation}\label{sunday1}
		\fourmatrix{-t}{s}{1}{0} \binom{x}{y} Q_1 + \binom{x}{y}Q_2 = 0. 
	\end{equation}
	Note that the matrix $J = \sfourmatrix{1}{t}{0}{-1}$ satisfies 
$\sfourmatrix{0}{s}{1}{t}J = J\sfourmatrix{t}{-s}{-1}{0}$. Set $\binom{u}{v} = J\binom{x}{y}Q_1$; we will argue that $u + \omega v \ne 0$ and $(u + \omega v)(A + \omega B) = 0$, which contradicts $\rank(A+\omega B) = d$. If $u + \omega v = 0$, then $\binom{u}{v} = 0$ and, since $J$ is nonsingular, $\binom{x}{y}Q_1 = 0$. This implies $xQ_1 = yQ_1 = 0$, which together with (\ref{sunday1}) and the fact that $\rank(Q_1 | Q_2) = h$ yields $\binom{x}{y} = 0$, which is not the case. Therefore $\binom{u}{v} \ne 0$. We have $(u+ \omega v)(A + \omega B) = (uA + svB) + \omega(uB + vA + tvB) = \left\langle\binom{1}{\omega},\binom{u}{v}A + \sfourmatrix{0}{s}{1}{t}\binom{u}{v}B\right\rangle.$ Now 
	\begin{align*}
		\binom{u}{v}A + \fourmatrix{0}{s}{1}{t}\binom{u}{v}B &= J\binom{x}{y}Q_1 A + \fourmatrix{0}{s}{1}{t}J\binom{x}{y}Q_1B \\ 
		&= J\left(\binom{x}{y}Q_1A + \fourmatrix{t}{-s}{-1}{0}\binom{x}{y}Q_1B\right) \\
		&= -J\left(\binom{x}{y}Q_2 + \fourmatrix{-t}{s}{1}{0}\binom{x}{y}Q_1\right)B,
	\end{align*}
	since $Q_1A = -Q_2B$. Now combining the above with (\ref{sunday1}) we see that $(u + \omega v)(A + \omega B) = 0$, contradicting the fact that $\rank(A + \omega B) = d$ and $u + \omega v \ne 0$. 
	
\end{proof}

The above lemma has the following as a straightforward corollary. 

\begin{lemma}\label{sunday}
	Let $\bF = \bF_0(\omega)$ be a degree-$2$ extension field of a field $\bF_0$. Let $h,d,m,n \in \bZ_0^+$ satisfy $0 \le h \le d \le n$ and $A,B \in \bF_0^{d \times n}$ and $P \in \bF_0^{m \times n}$ be such that $\rank\binom{A+ \omega B}{P} = m+d$, $\rank(P) = m$ and $\rank\left(\begin{smallmatrix}A \\ B \\ P \end{smallmatrix}\right) \le m+2d-h$. There exist matrices $A',B' \in \bF_0^{d \times n}$ such that $\binom{A + \omega B}{P}$ and $\binom{A' + \omega B'}{P}$ are row-equivalent and $B'$ has $h$ zero rows. 
\end{lemma}

\section{Examples}

We now investigate the two classes of $\GF(q)$-regular matroids from our main theorem. We define them differently from in the introduction in order to prove that they are both well-defined and $\GF(q)$-regular. We will use the fact that projective geometries are \emph{modular}; that is, that every pair of flats $F_1,F_2$ satisfies $r(F_1 \cap F_2) = r(F_1) + r(F_2) - r(F_1 \cup F_2)$.

 Let $\bF$ be a field with a $\GF(q)$-subfield, $n \ge 3$ be an integer, $A \in \cP\cG(n-1,q)$ and $N = \tilde{M}(A) \cong \PG(n-1,q)$. Let $L_0$ be a line of $N$ and $v \in \col_{\bF}(A[L_0])$ be not parallel to any column of $A[L_0]$. Let $f \in E(N)-L_0$ and $\cL$ be the collection of lines of $\cl_N(L_0 \cup \{f\})$ not containing $f$, noting that $|\cL| = q^2$. For each $L \in \cL$, let $v_L$ be a nonzero vector in the rank-$1$ subspace $\col_\bF(A_L) \cap \col_{\bF}(v | A[f])$. Let $X = \{x_L : L \in \cL\}$ be a $q^2$-element set and let $\overline{A} \in \bF^{[n] \times (E(N) \cup X)}$ be the matrix so that $\overline{A}[E(N)] = A$ and $\overline{A}[x_L] = v_L$ for each $L \in \cL$.

\begin{lemma}
	The matroid $\tilde{M}(\overline{A})$ is determined up to isomorphism by the choice of $n$ and $q$. 
\end{lemma}
\begin{proof}
	Let $M = \tilde{M}(\ol{A})$. We have $M \del X = N \cong \PG(n-1,q)$. Let $\cF_N$ be the set of cyclic flats of $N$ and $\cF_M$ be that of $M$. Let $P = \cl_N(L_0 \cup \{f\})$. Note that every pair of lines of $P$ intersect. It is easy to check the following claim:
	
	\begin{claim}\begin{align*}
			\cF_M = \ \  & \{F: F \in \cF_N, |F \cap P| \le 1\} \\
			 \cup & \{F \cup X:  F \in \cF_N, F \cap P = \{f\}\} \\
			 \cup & \{F \cup \{x_{L}\}: F \in \cF_N, F \cap P = L \in \cL\} \\
			 \cup & \{F : F \in \cF_N, r_M(F \cap P) = 2, F \cap P \notin \cL\} \\
			 \cup & \{F \cup X: F \in \cF_N, P \subseteq F\}  .\end{align*} 
	\end{claim}
		
	Since a matroid is determined by its collection of cyclic flats, the matroid $\tilde{M}(\ol{A})$ is therefore determined, for a given $n$ and $q$, by the naming of elements in $X$ and the choice of $N,P$ and $f$. There is only one choice for $N$ up to isomorphism, and the lemma now follows from the fact that the $\aut(\PG(n-1,q))$ acts transitively on pairs $(P,f)$, where $P$ is a plane containing $f$. 
\end{proof}

	We write $\olpg(n-1,q)$ for any matroid isomorphic to $M(\overline{A})$. Note that $M = \olpg(n-1,q)$ arises from $N = \PG(n-1,q)$ by adding $q^2$ new points on a line, spanned by a plane $P$ of $M$ and spanning a single point of $P$. The following is immediate from the definition and the previous lemma. 
			
\begin{lemma}\label{lregular}
	The matroid $\olpg(n-1,q)$ is $\GF(q)$-regular. 
\end{lemma}

We now turn to our second class, which is simpler to analyse. Let $\bF$ be a field with a $\GF(q)$-subfield and let $n \ge 2$. Let $B \in \cPG(n,q)$ and $N = \tilde{M}(B)$. Let $L_0$ be a line of $N$ and $v \in \col_{\bF}(B[L_0])$ be a nonzero vector, not parallel to any column of $B[L_0]$. Let $e \notin E(N)$ and $B^+ \in \bF^{[n+1] \times (E(N) \cup \{e\})}$ be such that $B^+[E(N)] = B$ and $B^+[e] = v$. 

By modularity of $N$, the matroid $\tilde{M}(B^+)$ is isomorphic to the principal extension of $L_0$ in $N$ by the element $e$, and is therefore determined up to isomorphism by $n$ and $q$ (due to transitivity of $\aut(\PG(n,q))$ on its set of lines). We write $\hpg(n-1,q)$ for any matroid isomorphic to the rank-$n$ matroid $\si(\tilde{M}(B^+) \con e)$. The following is clear by construction:

\begin{lemma}\label{hregular}
	The matroid $\hpg(n-1,q)$ is $\GF(q)$-regular. 
\end{lemma}

While we have specified these matroids abstractly to emphasise their $\GF(q)$-regularity and the fact that they are well-defined, we will only be interested in their $\GF(q^2)$-representations. We first consider $\bpg(n-1,q)$. The line $X$ we add is a $U_{2,q^2+1}$-restriction spanned by an element $f$ of $N$, together with an element $x_{L_0}$ that is spanned by $L_0$ but not contained in $L_0$. Since there are at most $q^2+1$ points on every line in $\PG(n-1,q^2)$, there is only one way to add the points in $X$ given a choice of $f$ and $x_{L_0}$. By choosing a basis for $\GF(q^2)^n$ in which $L_0$ and $f$ correspond to the first three standard basis vectors, we see that $\bpg(n-1,q)$ has the following as a representation: 
\begin{align*}
 \ol{A}(n-1,q) = \kbordermatrix{
 	& x_{L_0} & X - \{x_{L_0}\} & E(N) \\
	& 	1	  &    	\alpha 		  &     \          \\
	&  \omega & 		\omega \alpha 		  & 	   \			 \\
	&   0     &      1         &     A          \\
	&   0     &     0 		  &    \           \\
	&   \vdots     &     \vdots  & & \  \\
	},
\end{align*}
where $\alpha$ ranges over $\GF(q^2)-\{0\}$, and $A \in \cPG(n-1,q)$ is such that $A_f$ is the third standard basis vector. 

Now we consider $\hpg(n-1,q)$. Let $B \in \cPG(n,q)$ be a matrix containing among its columns the standard basis vectors $b_1, \dotsc, b_{n+1} \in \GF(q)^{n+1}$. If we choose $L_0$ to be the line spanned by $b_1$ and $b_2$ and $v$ to be the vector $b_1 - \omega b_2$, the matroid $\hpg(n-1,q)$, obtained by appending $v$ to $B$ and contracting the corresponding element, has the following representation: 
\[
	\widehat{A}(n-1,q) = \left(\begin{array}{ccccc}	 
	 (0 + 0\omega)\mathbf{j} & (1 + 0\omega) \mathbf{j} & \dotsc & (s+t \omega)\mathbf{j} & \dotsc \\
	    A     &    A    & \dotsc &      A    & \dotsc    \\
	\end{array}\right), 
\]
where $A \in \cPG(n-2,q)$, $\mathbf{j} = (1, \dotsc, 1)$  denotes the all-ones vector with $\frac{q^{n-1}-1}{q-1}$ entries, and $s$ and $t$ range over $\GF(q)$. Note that every vector in $\GF(q^2)^n$ with all but the first entry in $\GF(q)$ is parallel to a column of $\ha(n-1,q)$.

We have defined $\hpg(n-1,q)$ and $\olpg(n-1,q)$ abstractly, not as $\GF(q^2)$-represented matroids. When we refer to the associated $\GF(q^2)$-represented matroids we will write $M(\ha(n-1,q))$ and $M(\ba(n-1,q))$. 

\section{Non-examples}\label{nonexamples}

Let $\bF = \bF_0(\omega)$ be a degree-$2$ extension field of a field $\bF_0$. For a vector $w \in \bF^t$, we write $L(w)$ for the subspace $\spn_{\bF_0}(\{u,v\})$, where $u$ and $v$ are the unique $\bF_0$-vectors so that $w = u + \omega v$. Note that $L(w)$ has dimension $2$ if and only if $w$ is not parallel to an $\bF_0$-vector. 

We now define an important class of rank-$3$ represented matroids that will serve as obstructions to $\GF(q^2)$-regularity. Let $\cO(q)$ denote the set of $\GF(q^2)$-represented matroids $M$ such that $M \approx M(A \ |\ G_3)$, where the column set $X$ of $A$ has three elements, $G_3 \in \cPG(2,q)$, and $A \in \GF(q^2)^{[3] \times X}$ is a rank-$3$ matrix such that the three subspaces $L(A_x): x \in X$ each have dimension $2$ and together have trivial intersection. 

	More geometrically, if $M \in \cO(q)$ then $\tilde{M}$ is obtained by extending a projective plane $R$ over $\GF(q)$ by a three-element independent set $X$ so that $\tilde{M}$ is $\GF(q^2)$-representable and there is no point of $R$ common to the three lines of $R$ spanning the three points of $X$. 




\begin{lemma}\label{obad}
	If $M \in \cO(q)$, then $\tilde{M}$ is representable over a field $\bF$ if and only if $\bF$ has $\GF(q^2)$ as a subfield. 
\end{lemma}
\begin{proof}
	Let $M \in \cO(q)$ and $X,A,G_3$ be defined as above. Let $X = \{x_1,x_2,x_3\}$ and $R = M \del X$, noting that $\tilde{R} \cong \PG(2,q)$.  Each pair of subspaces in $\{L(A_x): x\in X\}$ meet in dimension $1$; let $e_i$ be the unique element of $E(R)$ so that $G_3[e_i] \in \cap_{j \in [3]-\{i\}}(L(x_j))$. Moreover by Lemma~\ref{gfqvectorinplane} each pair of columns of $A$ spans a nonzero $\GF(q)$-vector; for each $i \in [3]$ let $f_i$ be the unique element of $E(R)$ so that $G_3[f_i] \in \col(A[X-\{x_i\}])$. Note that $\tilde{M}$ is a simple rank-$3$ matroid, that $\tilde{R} \cong \PG(2,q)$, and that the subspaces $L(A_x): x \in X$ correspond to three lines $L_1,L_2,L_3$ of $\tilde{R}$ so that $x_i \in \cl_{\tilde{M}}(L_i)$ and $L_1 \cap L_2 \cap L_3 = \varnothing$. 	Further observe that if $i,j \in [3]$ and $i \ne j$, then $f_i \notin L_j$. 
	Since $\tilde{M}$ is $\GF(q^2)$-representable it is also representable over all fields with a $\GF(q^2)$-subfield, so it remains to show that $\tilde M$ is not representable over any other fields. 
	
	Let $\bF$ be a field over which $\tilde{M}$ is representable and assume for a contradiction that $\bF$ does not have a $\GF(q^2)$-subfield. Since $\tilde{R}$ is a minor of $\tilde{M}$ it follows that $\bF$ has $\GF(q)$ as a subfield. Let $P \in \bF^{[3] \times E(M)}$ be a $\bF$-representation of $\tilde{M}$; by Lemma~\ref{confinement} we may assume that $P[E(R)]$ is a $\GF(q)$-matrix and by applying further $\GF(q)$-row operations and $\GF(q^2)$-column scalings we may assume (using the fact that $f_i \notin L_j$ for $i \ne j$) that $P$ has the form 	
	\begin{align*}
 P = \kbordermatrix{
 	& e_1 & e_2 & e_3 & x_1       & x_2     & x_3       & f_1 & f_2 & f_3  &  & \\
	& 1   & 0   & 0   &  0        & \alpha_2 & \alpha_3 & s_1 & 1   & 1    & \ & \\
	& 0   & 1   & 0   &  1        & 0        & 1        & 1   & s_2 & s_4  & \dotsc & \\
	& 0   & 0   & 1   &  \alpha_1 & 1        & 0        & 1   & s_3 & s_5  & \ & \\
	},
	\end{align*}
where $\alpha_i \in \bF - \GF(q)$ for each $i \in [3]$, $s_1 \in \{0,1\}$ and $s_j \in \GF(q)$ for each $j \in [5]$. Since $r_{\tilde{M}}(x_2,x_3,f_1) = 2$, we have $\alpha_2 + \alpha_3 = s_1$. The lines $\cl_{\tilde{M}}(\{f_2,x_3\})$ and $\cl_{\tilde{M}}(\{x_3,f_2\})$ both intersect $L_2$ at $x_1$, so the vectors $(0,1-\alpha_3 s_2,-\alpha_3s_3)$ and $(0,-\alpha_2s_4,1-\alpha_2s_5)$ are both parallel to $(0,1,\alpha_1)$ and thus $\alpha_3s_3\alpha_2s_4 = (1-\alpha_3s_2)(1-\alpha_2s_5)$. Using $\alpha_3 = s_1 -\alpha_2$, we see that $\alpha_2$ is a zero of the function 
\[ p(z) = s_3s_4z(s_1-z) - (1-s_2(s_1-z))(1-s_5z).\]
Now $p(z)$ is a polynomial in $z$ with coefficients in $\GF(q)$ and degree at most $2$. However, $\alpha_2 \notin \GF(q)$ and, since $\bF$ has no $\GF(q^2)$-subfield, $\alpha_2$ is not a zero of an irreducible quadratic over $\GF(q)$. Therefore $p(z)$ is identically zero. We have $0 = p(0) = 1-s_1s_2$, so $s_1s_2 = 1$; since $s_1 \in \{0,1\}$ this gives $s_1 = s_2 = 1$. Similarly we have $0 = p(s_1) = s_5 - 1$, so $s_5 = 1$. Therefore $p(z) = z(1-z)(s_3s_4-1)$, so $s_3s_4 = 1$. 
Let $s_3 = t$ and $s_4 = t^{-1}$. Since $r_{\tilde{M}}(\{x_1,x_3,f_2\}) = r_{\tilde{M}}(\{x_1,x_2,f_3\}) = 2$, we have $\alpha_1 = t(1 + \alpha_2^{-1})$ and $\alpha_1 + (1-\alpha_2)^{-1} = t$. A computation gives $\alpha_2 = (1+t)^{-1}$, contradicting $\alpha_2 \notin \GF(q)$.
\end{proof} 

We now precisely determine the matrices $A$ which, when appended to a matrix in $\cPG(t-1,q)$, yield a matroid with no $\cO(q)$-minor; these matrices are all essentially restrictions of $\ha(t-1,q)$ and $\ba(t-1,q)$. We also give an alternative characterisation of these matrices in terms of the subspaces $L(x)$ defined as above. This is equivalent to a treatment of the special case of our main theorem where $M$ has a spanning projective geometry restriction.  

	\begin{lemma}\label{magic}
	Let $q$ be a prime power, $t \ge 3$ be an integer and $G_t \in \cPG(t-1,q)$. If $A \in \GF(q^2)^{[t] \times Y}$ and $M = M(A\ |\ G_t)$ then the following are equivalent:
	\begin{enumerate}
		\item\label{m1} $M$ has a minor in $\cO(q)$;
		\item\label{m2} $\si(M)$ is not projectively equivalent to a restriction of either $M(\ha(t-1,q))$ or $M(\ba(t-1,q))$;
		\item\label{m3} there exists a set $Z \subseteq Y$, independent in $M$, such that $|Z|\in \{2,3\}$ and the subspaces $L(A_z): z \in Z$ each have dimension $2$ and have trivial intersection.				
	\end{enumerate}
	Moreover, if $t \ge 5$ and (\ref{m3}) is satisfied by a set $Z$ of size $2$, then the matroid $M \con Z \del (Y-Z)$ also has a minor in $\cO(q)$. 
\end{lemma}
	We call a matrix $A$ satisfying the conditions in this lemma \emph{$q$-bad} and if (\ref{m3}) holds with $|Z| = 2$ we call $A$ \emph{strongly $q$-bad}. Note that property (\ref{m3}), and therefore (strong) $q$-badness, is invariant under $\GF(q)$-row equivalence.

\begin{proof}[Proof of Lemma~\ref{magic}:]
	 Let $b_1, \dotsc, b_t$ be the standard basis vectors of $\GF(q)^t$.  We showed in Lemmas~\ref{lregular} and~\ref{hregular} that $\hpg(n-1,q)$ and $\olpg(n-1,q)$ are $\GF(q)$-regular and in Lemma~\ref{obad} that the matroids in $\cO(q)$ are not, so (\ref{m1}) implies (\ref{m2}). 
	
	Suppose that (\ref{m2}) holds. Note that (\ref{m3}) and its negation are invariant under $\GF(q)$-row-equivalence. Let $Y' = \{y \in Y, \dim(L(A_y)) = 2\}$ and $\cL = \{L(A_y): y \in Y'\}$, noting that every $y \in Y - Y'$ is a loop or is parallel to some column of $G_t$, so $\si(M \del (Y-Y')) \cong \si(M)$. If there exist $z_1,z_2 \in Y'$ such that $L(A_{z_1})$ and $L(A_{z_2})$ are skew then $Z = \{z_1,z_2\}$ satisfies (\ref{m3}), so we may assume that $Y'$ contains no such pair.
	
	If all subspaces in $\cL$ have a dimension-$1$ subspace in common, then, by applying $\GF(q)$-row-operations, we may assume that this subspace is $\spn_{\GF(q)}(b_1)$. This gives a matrix representation of $\si(M)$  that is, up to column scaling, a submatrix of $\ha(t-1,q)$, contradicting (\ref{m2}). We may therefore assume that $\bigcap \cL$ is trivial. 
		
	Therefore no pair of subspaces in $\cL$ are orthogonal but there is no dimension-$1$ subspace common to all subspaces in $\cL$. It follows routinely that there is some dimension-$3$ subspace $P$ of $\GF(q)^t$ containing all subspaces in $\cL$, so $r_M(Y') \le 3$. 
	
	If $r_M(Y') \le 2$ then there is a dimension-$2$ subspace $L_0$ of $\spn_{\GF(q^2)}(P)$ containing $A[Y']$. By Lemma~\ref{gfqvectorinplane}, $L_0$ contains a nonzero $\GF(q)$-vector $v$. Let $\{v,w\}$ be a basis for $L_0$. After $\GF(q)$-row-operations we may assume that $\{b_1,b_2,b_3\}$ is a basis for $P$, that $v = b_3$, and that $w \in \cl_{\GF(q^2)}(\{b_1,b_2\})-\cl_{\GF(q^2)}(b_2)$. Moreover, after row-scalings over $\GF(q^2)$ we may assume that either $w = b_1$ or $w = b_1 + \omega b_2$. Since $r_M(Y') = 2$ it follows that $\si(M)$ is projectively equivalent to a restricition of $\ha(t-1,q)$ or $\ba(t-1,q)$, contradicting (\ref{m2}). 
	
	If $r_M(Y') = 3$ then let $Z = \{z_1,z_2,z_3\}$ be a basis for $Y'$. Let $L_i = L(A_{z_i})$ for each $i \in \{1,2,3\}$. Since $r_M(Z) = 3$, the lines $L_1,L_2,L_3$ are not all equal, so we may assume that $L_1 \notin \{L_2,L_3\}$. If $L_1,L_2,L_3$ have no dimension-$1$ subspace in common then (\ref{m3}) holds, so we may assume that $L_1 \cap L_2 \cap L_3$ has dimension $1$. Moreover we know that there is some other subspace $L_4 = L(A_{z_4}) \in \cL$ not containing $L_1 \cap L_2 \cap L_3$, as $\bigcap \cL$ is trivial. Now $L_1 \cap L_2 \cap L_4$ and $L_1 \cap L_3 \cap L_4$ are both trivial, and either $\{z_1,z_2,z_4\}$ or $\{z_1,z_3,z_4\}$ has rank $3$ in $M$. Therefore (\ref{m3}) holds. 
		
	Finally, suppose that (\ref{m3}) holds. If $|Z| = 2$ then let $Z = \{z_1,z_2\}$. By applying $\GF(q)$-row-operations if necessary we may assume that $L(z_1) = \spn_{\GF(q)}(\{b_1,b_2\})$ and $L(z_2) = \spn_{\GF(q)}(\{b_3,b_4\})$. Let $X$ be the set of columns of $G_t$ contained in $\spn_{\GF(q)}(L(z_1) \cup L(z_2))$ and $N = M|(X \cup \{z_1,z_2\})$. We have  
	\begin{align*}
	 N \approx M \kbordermatrix{
 	&  z_1 & z_2 & X \\
	&  1   & 0   &   & \\
	& \alpha_1 & 0 & \dotsc  \\
	& 0 & 1 &       \dotsc  \\
	& 0 & \alpha_2 &   \\
	},
	\end{align*}
	for some $\alpha_1, \alpha_2 \in \GF(q^2)-\GF(q)$, where the matrix contains exactly one column from each parallel class in $\GF(q)^4$. Therefore, $N \con z_1$ is represented by a matrix having a submatrix containing as columns at least one nonzero vector from each parallel class of $\GF(q)^3$, as well as columns parallel to $(0,1,\alpha_2)^T,(-\alpha_1,1,0)^T$ and $ (-\alpha_1,0,1)^T$. Restricting $N \con z_1$ to this submatrix yields a matroid in $\cO(q)$. Moreover, if $t \ge 5$ then let $X'$ be the set of columns of $t$ contained in $\spn_{\GF(q)}(L(z_1) \cup L(z_2) \cup \{t_5\})$ and let $N' = M|(X' \cup \{z_1,z_2\})$. It is easy to see by a similar argument to the above that $N' \con \{z_1,z_2\}$, which is a restriction of $M \con Z \del (Y-Z)$, has a spanning restriction in $\cO(q)$. 
	
	If (\ref{m3}) holds for some $Z$ of size $3$ but for no $2$-element subset of $Z$, then $Z$ contains three dimension-$2$ subspaces, all contained in a common dimension-$3$ subspace, with trivial intersection. This dimension-$3$ subspace corresponds to a plane $P$ of the spanning $\PG(t-1,q)$-restriction of $M$, and clearly $M|(P \cup Z) \in \cO(q)$.  
\end{proof}

	\section{Tangles}\label{tanglesec1}
	
	Our tool for constructing minors in $\cO(q)$ given a projective geometry minor (rather than a spanning restriction as in Lemma~\ref{magic}) is the \emph{tangle}. Tangles were introduced for graphs, and implicitly for matroids, by Robertson and Seymour [\ref{gmx}] and were later extended explicitly to matroids [\ref{d95},\ref{ggrw06}]. The techniques in this section and the next follow [\ref{ggwnonprime}].
	
	Let $M$ be a matroid and let $\theta \in \bZ^+$. A set $X \subseteq E(M)$ is \emph{$k$-separating in $M$} if $\lambda_M(X) < k$. A collection $\cT$ of subsets of $E(M)$ is a \emph{tangle of order $\theta$} if 
	\begin{enumerate}
		\item Every set in $T$ is $(\theta-1)$-separating in $M$ and, for each $(\theta-1)$-separating set $X \subseteq E(M)$, either $X \in T$ or $E(M) - X \in \cT$; 
		\item if $A,B,C \in \cT$ then $A \cup B \cup C \ne E(M)$; and
		\item $E(M) - \{e\} \notin \cT$ for each $e \in E(M)$. 
	\end{enumerate}
	We refer to the sets in $\cT$ as \emph{$\cT$-small}. Given a tangle of order $\theta$ on a matroid $M$ and a set $X \subseteq E(M)$, we set $\kappa_{\cT}(X) = \theta - 1$ if $X$ is contained in no $\cT$-small set, and $\kappa_{\cT}(X) = \min\{\lambda_M(Z): X \subseteq Z \in \cT\}$ otherwise. The proof of our first lemma appears in [\ref{ggrw06}]:
	
	\begin{lemma}\label{tanglematroid}
		If $\cT$ is a tangle of order $\theta$ on a matroid $M$, then $\kappa_{\cT}$ is the rank function of a rank-$(\theta-1)$ matroid on $E(M)$. 
	\end{lemma}
	
	This matroid, which we denote $M(\cT)$, is the \emph{tangle matroid}. The next lemma is easily proved:
	
	\begin{lemma}\label{tangleminor}
		If $N$ is a minor of a matroid $M$ and $\cT_N$ is a tangle of order $\theta$ on $N$, then $\{X \subseteq E(M): \lambda_M(X) < \theta-1, X \cap E(N) \in \cT_N\}$ is a tangle of order $\theta$ on $M$.
	\end{lemma}
	
	This tangle is the tangle on $M$ \emph{induced} by $\cT_N$. 
	
	If $M$ is a matroid and $k$ is an integer, then we write $\cT_k(M)$ for the collection of $(k-1)$-separating sets of $M$ that are neither spanning nor cospanning. For example, if $M \cong \PG(n-1,q)$ and $n \ge k$, then $\cT_k(M)$ is simply the collection of subsets of $E(M)$ of rank at most $k-2$. Since $3\frac{q^{n-2}-1}{q-1} < \frac{q^n-1}{q-1}$, no three such subsets have union $E(M)$, and we easily have the following:

	\begin{lemma}\label{pgtangle}
		If $q$ is a prime power, $n \in \bZ^+$, and $M \cong \PG(n-1,q)$, then $\cT_n(M)$ is a tangle of order $n$ in $M$. 
	\end{lemma}
	
	If $M$ is a matroid with a $\PG(n-1,q)$-minor $N$, then we write $\cT_n(M,N)$ for the tangle of order $n$ in $M$ induced by $\cT_n(N)$. 
	
	The next result is a slight variation of a lemma from [\ref{ggwnonprime}].
	
	\begin{lemma}\label{tanglecontract}
		Let $k \in \bZ^+$, let $M$ be a matroid and let $N$ be a minor of $M$ such that $\cT_k(N)$ is a tangle. If $X \subseteq E(M)$ is contained in a $\cT_k(M,N)$-small set, then there is a minor $M'$ of $M$ such that $M'|X = M|X$, $M'$ has $N$ as a minor, and $X$ is contained in a $\cT_k(M',N)$-small set $X'$ such that $E(M') = E(N) \cup X'$ and $\lambda_{M'}(X') = \kappa_{\cT_k(M',N)}(X) = \kappa_{\cT_k(M,N)}(X)$. 
	\end{lemma}
	\begin{proof}
		Let $b = r_{\cT_k(M,N)}(X)$ and let $M'$ be a minimal minor of $M$ such that $N$ is a minor of $M$, $M|X = M'|X$ and $r_{\cT_k(M',N)}(X) = b$. Let $\cT = \cT_k(M',N)$ and $X' = \cl_{M(\cT)}(X)$. It remains to show that $E(M') = X' \cup E(N)$. If not, there is some $e \in E(M') - X' \cup E(N)$. Since $\cl_{M'}(X) \subseteq X'$, we know that $M|X$ is a restriction of both $M \con e$ and $M \del e$. If $N$ is a minor of $M \con e$, and so by choice of $M$ we have $r_{\cT_k(M \con e,N)}(X) \le b-1$. Therefore there is some set $Z \in \cT_k(M \con e,N)$ such that $\lambda_{M' \con e}(Z) \le b-1$ and $X \subseteq Z$. Therefore $Z \cup \{e\} \in \cT$ and $\lambda_{M'}(Z \cup \{e\}) \le b$ so $r_{\cT}(X \cup \{e\}) = r_{\cT}(X)$ and $e \in \cl_{\cT}(X)$, a contradiction. The case where $N$ is a minor of $M \del e$ is similar. 
	\end{proof}
	
	\section{Using a Tangle}\label{tanglesec2}
	Our first lemma allows us to find an affine geometry restriction in a dense $\GF(q)$-representable matroid $M$ after contracting a subset of an arbitrary set of bounded size. A stronger qualitative version of this lemma (in which such a restriction is found in $M$ itself) follows from the density Hales-Jewett theorem [\ref{fk91}], but the proof of this result is much easier and we obtain a constructive bound.  
	
	\begin{lemma}\label{densegfq}
		Let $\alpha \in \bR^+$, $q$ be a prime power, and $n,h,k \in \bZ^+$ satisfy $n \ge (2+k)h + \log_q(2/\alpha)$ and $k \ge 2q^h(1/\alpha-1)$. If $M$ is a rank-$r$ $\GF(q)$-representable matroid with $r \ge n$ and $\elem(M) \ge \alpha|\PG(r-1,q)|$ then for each rank-$hk$ independent set $C$ in $M$, there exists $C' \subseteq C$ such that $M \con C'$ has an $\AG(h,q)$-restriction.
	\end{lemma}
	\begin{proof}
		Let $(C_1,C_2, \dotsc, C_k)$ be a partition of $C$ into sets of size $h$, and for each $i \in \{0, \dotsc, k\}$ let $M_i = M \con (C_1 \cup \dotsc \cup C_i)$ and $\delta_i = \elem(M_i)/|\PG(r(M_i)-1,q)|$, noting that $\delta_0 \ge \alpha$ and $\delta_i \le 1$ for each $i$. Let $x = \tfrac{1}{2}q^{-h}$ and let $j$ be maximal such that $j \le k$ and $\delta_j \ge \alpha(1+x)^j$. If $j = k$ then we have $\delta_k \ge \alpha(1+x)^k > \alpha(1+kx) \ge 1,$ a contradiction. Therefore $j < k$, and we have $\delta_j \ge \alpha(1+x)^j$ and $\delta_{j+1} < \alpha(1+x)^{j+1}$. 
		
		Let $F = \cl_{M_j}(C_{j+1})$ and $\cF$ be the collection of rank-$(h+1)$ flats of $M_j$ containing $F$; we have $\elem(M_{j+1}) = |\cF|$ and $\elem(M_j) = \elem(M_j|F) + \sum_{H \in \cF}(\elem(M_j|H)-\elem(M_j|F))$. We may assume that $M_j|H \not\cong \AG(h,q)$ for each $H \in \cF$, and therefore that $\elem(M_j|H) - \elem(M_j|F) < q^h$ for each $H \in \cF$. Let $r = r(M_j) = n-hk$. Now
		\begin{align*}
			\alpha(1+x)^j\frac{q^r-1}{q-1} & \le \elem(M_j) \\
										 & = \elem(M_j|F) + \sum_{H \in \cF}(\elem(M_j|H)-\elem(M_j|F))\\
										 & \le \frac{q^h-1}{q-1} + (q^h-1)\elem(M_{j+1})\\
										 & < \frac{q^h-1}{q-1} + \alpha(q^h-1)(1+x)^{j+1}\frac{q^{r-h}-1}{q-1}.
		\end{align*}
		Simplifying this inequality gives
		\[x(q^r-1) + \frac{q^h-1}{(1+x)^j\alpha} > (1+x)(q^h + q^{r-h}-2),\]
		and so, using $x > 0$ and $q^h \ge 2$, we have $xq^r + q^h/\alpha > q^{r-h}$. This implies that $q^r < 2q^{2h}/\alpha$, contradicting $r \ge 2h + \log_q(2/ \alpha)$. 
	\end{proof}
	
	We now combine the previous lemma and the machinery of tangles to show that, given a small restriction of $M$ with given `connectivity' to a large projective geometry minor of $M$, we can realise the same connectivity to a projective geometry restriction in a minor of $M$. The `qualitative' version of this lemma, on whose proof ours is based, will appear in [\ref{ggwnonprime}]. 
	
	\begin{lemma}\label{tanglecontractnumbers}
		Let $q$ be a prime power, let $h,a \in \bZ^+$ satisfy $a \le h$ and let $n = 2h(1+q^{h+a})+a+2$. If $M$ is a matroid with a $\PG(n-1,q)$-minor $N$ and $X\subseteq E(M)$ is a set such that $r_M(X) \le a$ and $M \del X$ is $\GF(q)$-representable, then there is a minor $M'$ of $M$ and a $\PG(h-1,q)$-restriction $N'$ of $M'$ such that $E(M') = E(N') \cup X$, $M'|X = M|X$ and $\lambda_{M'}(X) = \kappa_{\cT_k(M,N)}(X)$.
	\end{lemma}
	\begin{proof}
		Let $k = 2q^{h+a}$ and $\alpha = (q^a + 1)^{-1}$, noting that $h,k,n$ and $\alpha$ satisfy the numerical conditions in Lemma~\ref{densegfq}. Let $b = \kappa_{\cT_n(M,N)}(X)$. By Lemma~\ref{tanglecontract} there is a minor $M_1$ of $M$ having $N$ as a minor and a $\cT_n(M_1,N)$-small set $X_1$ containing $X$ such that $E(M_1) = E(N) \cup X_1$ and $\lambda_{M_1}(X_1) = \kappa_{\cT_{n}(M_1,N)}(X) = b$.
		
		Note for each independent set $C$ of $N$ that $\cT_{n-|C|}(N\con C)$ is a tangle of order $n - |C|$ on $N \con C$. Let $C$ be a maximal independent set of $N \del (X \cap E(N))$ so that 
		\begin{enumerate}
			\item $|C| \le hk$,
			\item $M_1|X = (M_1 \con C)|X$, and
			\item \label{c3}$\kappa_{\cT_{n-|C'|}(M_1 \con C',N \con C')}(X) = b$ for all $C' \subseteq C$.
		\end{enumerate}
		Let $M_2 = M_1 \con C$, $N_2 = N \con C$, $\cT = \cT_{n-|C|}(M_2,N_2)$ and $X' = \cl_{M(\cT)}(X)$.
		\begin{claim}
			$|C| = hk$.
		\end{claim} 
		\begin{proof}[Proof of claim:]
			Suppose that $|C| \le hk-1$. Since $\kappa_{\cT}(X') = b \le n - hk < n - |C|$, we have $X' \in \cT$, so $E(N_2) - X'$ is spanning in $N_2$. Further note that $r_{M_2}(X) = a < n - |C|$; let $e \in E(N_2) - X' - \cl_{M_2}(X)$. By choice of $C$ and $e$, we may assume that $X$ has rank at most $b-1$ in $\cT_{n  - |C' \cup \{e\}|}(M_2 \con e,N_2 \con e)$ for some $C' \subseteq C$, so there is some set $Z$ such that $C' \cup \{e\} \subseteq Z$, $\lambda_{M_2 \con e}(Z) \le b-1$ and $Z \cap E(N_2 \con e)$ is not spanning in $N_2 \con e$. Therefore $(Z \cup e) \cap E(N_2)$ is not spanning in $N_2$ and $\lambda_{M_2}(Z \cup \{e\}) \le b$. It follows that $e \in \cl_{\cT}(X) = X'$, a contradiction. 			
		\end{proof}
		Since $X_1 \cap E(N)$ is not spanning in $N$ and $N$ is round, it follows that $r_{N}(X_1 \cap E(N)) = \lambda_{N}(X_1 \cap E(N)) \le \lambda_{M_1}(X_1) = b$. Therefore $n \le r(M_1|E(N)) \le n + b$. Now 
		\begin{align*}
			\elem(M_1 \del X_1) &\ge \frac{q^n-1}{q-1} - \frac{q^b-1}{q-1}\\ &\ge (q^b+1)^{-1}\frac{q^{n+b}-1}{q-1}\\ & \ge \alpha|\PG(r(M_1|E(N))-1,q)|.
		\end{align*}
	The matroid $M_1|E(N)$ is a minor of $M \del X$ and is therefore $\GF(q)$-representable. Moreover, $C$ is an $hk$-element independent subset of $E(N)$, so by Lemma~\ref{densegfq} there is a set $C' \subseteq C$ such that $(M_1|E(N)) \con C'$ has an $\AG(h,q)$-restriction $(M_1 \con C')|A$.  Let $\cT' = \cT_{n-|C'|}(M_1 \con C', N \con C')$. Now $N \con C'$ is $\GF(q)$-representable and $\elem((N \con C')|A) = q^h$, so $r_{(N \con C')|A} \ge h+1 > b$. Therefore $\kappa_{\cT'}(A) \ge \kappa_{\cT_{n-|C'|}(N \con C')}(A) \ge b$. It follows that $\kappa_{M_1 \con C'}(X,A) = b$, as otherwise $M_1 \con C'$ has a $b$-separation for which neither side is $\cT'$-small. 
	
	By Theorem~\ref{linking}, there is a minor $M_1'$ of $M_1 \con C'$ with $E(M_1') = X \cup A$, $M_1'|X = (M_1 \con C')|X = M|X$, $M_1'|A = (M_1 \con C')|A \cong \AG(h,q)$ and $\lambda_{M_1'}(X) = b$. Since $r(M_1'|A) =h+1 > b$, there is some $e \in A - \cl_{M_1'}(X)$. Contracting $e$ and simplifying yields the required minor $M'$. 
	\end{proof}
	
	Note in the above lemma that, in the special case where $M$ is round we have $\kappa_{\cT_k(M,N)}(X) = r_M(X)$; it follows that $N'$ is spanning in $M'$.

\section{Augmenting Structure}
	
	We now consider a matroid $M$ and an element $e \in E(M)$ such that $\si(M \con e)$ is a restriction of $\hpg(r(M)-2,q)$ or $\olpg(r(M)-2,q)$; we essentially argue that $M$ itself either has one of these two structures, or satisfies some constructive condition certifying otherwise. Unfortunately these hypotheses and outcomes are somewhat opaque in the two lemmas that follow; Theorem~\ref{maintech} will unify them. 
	
	We consider a slight variation of contraction in this section for ease of notation. If $e$ is a nonloop of a represented matroid $M$, then we let $M \dcon e$ denote the represented matroid $M' \con e'$, where $M'$ is obtained from $M$ by extending $e$ in parallel by an element $e'$. Thus, $e$ is a loop of $M \dcon e$, and we have $M \con e = (M \dcon e) \del e$ and $E(M \dcon e) = E(M)$. Note that if $M \dcon e \approx M(A)$ for some $\bF$-matrix $A$, then $M \approx M(A')$ for some matrix $A'$ obtained by appending a single row to $A$. 

	\begin{lemma}\label{haugment}
		Let $\bF = \bF_0(\omega)$ be a degree-$2$ extension field of a field $\bF_0$. Let $M$ be a vertically $5$-connected $\bF$-represented rank-$r$ matroid and $e$ be a nonloop of $M$ such that $M \dcon e \approx M\binom{u_0 + \omega v_0}{R}$ for some $u_0,v_0 \in \bF_0^{E(M)}$ and $R \in \bF_0^{[r-2] \times E(M)}$. Then there are matrices $P,Q \in \bF_0^{[2] \times E(M)}$ such that $M \approx M\binom{P + \omega Q}{R}$ and either
		\begin{enumerate}
			\item\label{a1} there is a partition $(I,J)$ of $E(M)$ such that \[\rank(R[I]) + \rank(Q[J]) \le 1,\] or 
			\item\label{a3} the matrix \vspace{-0.3cm}
			\[ W^+ = \kbordermatrix{
				& &\!\!\!S\!\!\!   &   & X & E(M) \\ 
				 \sqbr{2}& I_2  && 0 & -\omega I_2 & P\\
				 \sqbr{2}& 0   && I_2 & I_2 & Q\\
				 \sqbr{r-2}& 0   && 0   & 0   & R \\ }\]
			satisfies $\kappa_{M(W^+)}(S \cup X,K) = 4$ for every set $K \subseteq E(M)$ such that $r_M(K) \ge 4$. (Here $|S| = 4$ and $|X| = 2$.)

		\end{enumerate}
	\end{lemma}
	\begin{proof}
		Since $M \dcon e \approx M\binom{u_0 + \omega v_0}{R}$, we have $M \approx M\binom{P_1 + \omega Q_1}{R}$ for some $P_1,Q_1 \in \bF_0^{[2] \times E(M)}$. Let $W^+$ be the matrix in (\ref{a3}) with $P,Q = P_1,Q_1$ and let $M^+ = M(W^+)$. Note that $M \approx M^+ \con X \del S$ and $r(M^+) = r+ 2$. If (\ref{a3}) does not hold for $P_1,Q_1$, then there are sets $Z,K \subseteq E(M^+)$ such that $r_M(K) \ge 4$,  with $S\cup X \subseteq Z \subseteq  E(M^+) - K$ and $\lambda_{M^+}(Z) \le 3$. Let $(I,J) = (E(M) \cap Z,E(M)-Z)$.
		
		Note that $r_{M^+}(Z) \ge r_{M^+}(S) = 4$. We have $\lambda_{M}(I) \le \lambda_{M^+}(Z) \le 3$, so vertical $5$-connectivity of $M$ gives $\min(r_M(I),r_M(J)) \le 3$. But $r_M(J) \ge r_M(K) \ge 4$, so $r_M(I) \le 3$. This gives  $r_{M^+}(Z) \le 5$ and, by vertical $5$-connectivity of $M$, $r_M(J) = r$. 
		
		Note that $0 \le r_{M^+}(J) - r_M(J) \le r(M^+)-r(M) = 2$. We have $r = r_M(J) = \rank\left(\binom{P_1+\omega Q_1}{R}[J]\right)$ and $r_{M^+}(J) = \rank(W^+[J])$. By Lemma~\ref{sunday}, $\binom{P_1+ \omega Q_1}{R}[J]$ is row-equivalent to a matrix $\binom{P' + \omega Q'}{R[J]}$, where \[\rank(Q') = \rank(W^+[J])-\rank\left(\tbinom{P_1 + \omega Q_1}{R}[J]\right) = r_{M^+}(J)-r.\] Therefore $\binom{P_1+ \omega Q_1}{R}$ is row-equivalent to a matrix $\binom{P + \omega Q}{R}$ where $Q[J] = Q'$. Now $M = M\binom{P + \omega Q}{R}$ and
		\begin{align*}
			3 & \ge \lambda_{M^+}(Z) \\
			  & = r_{M^+}(Z) + r_{M^+}(J) - r(M^+) \\
			  & = (4 + \rank(R[I])) + (r + \rank(Q')) - (r + 2),\\
			  & = 2 + \rank(R[I]) + \rank(Q[J]) 
		\end{align*}
		so $\rank(R[I]) + \rank(Q[J]) \le 1$. Therefore (\ref{a1}) holds. 
	\end{proof}
	
	\begin{lemma}\label{laugment}
		Let $\bF = \bF_0(\omega)$ be a degree-$2$ extension field of a field $\bF_0$. Let $M$ be a rank-$r$, vertically $9$-connected $\bF$-represented matroid and $e$ be a nonloop of $M$. If there are matrices $P_0,Q_0 \in \bF_0^{[2] \times E(M)}$ and $R \in \bF_0^{[r-3] \times E(M)}$ and a partition $(I_0,J_0)$ of $E(M)$ such that $M \dcon e \approx M\binom{P_0 + \omega Q_0}{R}$, $r_{M \dcon e}(I_0) \le 2$, $\rank(R[I_0]) \le 1$ and $Q_0[J_0] = 0$, then there are matrices $P,Q \in \bF_0^{[3] \times E(M)}$ such that $M \approx M\binom{P+ \omega Q}{R}$ and either
		\begin{enumerate}
			\item\label{b0}
				$M$ and $e$ satisfy the hypotheses of Lemma~\ref{haugment},
			\item\label{b1}
				there is a partition $(I,J)$ of $E(M)$ such that $Q[J] = 0$ and $r_M(I) \le 4$, or
			\item\label{b2}
			the matrix \vspace{-0.3cm}
			\[ W^+ = \kbordermatrix{
				& &\!\!\!S\!\!\!   &   & X & E(M) \\ 
				 \sqbr{3}& I_3  && 0 & -\omega I_3 & P\\
				 \sqbr{3}& 0   && I_3 & I_3 & Q\\
				 \sqbr{r-2}& 0   && 0   & 0   & R \\ }\]
			satisfies $\kappa_{M(W^+)}(S \cup X,K) \ge 5$ for each set $K \subseteq E(M)$ such that $r_M(K) \ge 5$. (Here $|S| = 6$ and $|X| = 3$.)
		\end{enumerate}
	\end{lemma}
	\begin{proof}
		By hypothesis, there are matrices $P_1,Q_1 \in \bF_0^{[3] \times E(M)}$ such that $M \approx M\binom{P_1 + \omega Q_1}{R}$, where $P_1 = \binom{u}{P_0}$ and $Q_1 = \binom{v}{Q_0}$ for some vectors $u,v \in \bF_0^{E(M)}$. Let $W^+$ be the matrix in (\ref{b2}) with $P,Q = P_1,Q_1$ and let $M^+ = M(W^+).$ As before, we have $M \approx M^+ \con X \del S$, $r(M^+) = r+3$ and we may assume that there are sets $Z,K \subseteq E(M^+)$ with $r_M(K) \ge 5$ such that $S\cup X \subseteq Z \subseteq E(M) - K$ and $\lambda_{M^+}(Z) \le 4$. 
				
		Now $\lambda_{M}(E(M) \cap Z) \le \lambda_{M^+}(Z) \le 4$, so vertical $6$-connectivity of $M$ gives $\min(r_M(E(M) \cap Z),r(M \del Z)) \le 4$, but $r(M \del Z) \ge r_M(K) \ge 5$, so $r_M(E(M) \cap Z) \le 4$ and thus $r_{M^+}(Z) \le 7$ and $r_{M^+}(Z) \in \{6,7\}$. Let $F = \cl_{M^+}(Z)$, let $(I_1,J_1) = (E(M) \cap F, E(M)-F)$ and let $(I,J) = (I_0 \cup I_1,J_0 \cap J_1)$. 
		
		We have $r_M(I) \le (r_{M \dcon e}(I_0)+1)  + r_{M}(I_1) \le 3+4 = 7$, so by vertical $9$-connectivity of $M$ we get $r_M(J) = r$. Therefore $r_{M^+}(J) \ge r$. Moreover $r_{M^+}(J_1) = r(M^+) + \lambda_{M^+}(J_1) - r_{M^+}(F) \le (r+3) + 4 - r_{M^+}(F) = r+7-r_{M^+}(Z)$, so $r_{M^+}(J_1) \in \{r,r+1\}$. We consider the two cases separately. 
		
		If $r_{M^+}(J_1) = r$ then $r_{M^+}(J) = r$ and $W^+[J]$ is a rank-$r$ matrix with $(r+3)$ rows, so by Lemma~\ref{sunday}, $\binom{P_1+\omega Q_1}{R}[J]$ is row-equivalent to a matrix $\binom{P'}{R[J]}$ where $P' \in \bF_0^{[3] \times J}$. Therefore $\binom{P_1 + \omega Q_1}{R}$ is row-equivalent to a matrix $\binom{P + \omega Q}{R}$ where $Q[J] = 0$. Now $M \approx M\binom{P + \omega Q}{R}$ and $r_M(I) \le r_{M^+}(Z) - 3 \le 4$, so (\ref{b1}) holds.  
		
		If $r_{M^+}(J_1) = r+1$ then $r_{M^+}(F) = 6 = r_{M^+}(S)$ so $F = \cl_{M^+}(S)$. It follows that $R[I_1] = 0$. Also, $W^+[J_1]$ is a rank-$(r+1)$ matrix with $r+3$ rows, so by Lemma~\ref{sunday} the matrix $\binom{P_1 + \omega Q_1}{R}[J_1]$ is row-equivalent to a matrix $\binom{P' + \omega Q'}{R[J_1]}$ where $P',Q' \in \bF_0^{[3] \times J_1}$ and $Q'[J_1]$ has two zero rows. Therefore $\binom{P_1 + \omega Q_1}{R}$ is row-equivalent to a matrix $\binom{P + \omega Q}{R}$ where $P,Q \in \bF_0^{[3] \times E(M)}$ and $Q[J_1] = Q'$. Since $R[e] = 0$, it follows that $M \dcon e \approx M\binom{P_0' + \omega Q_0'}{R}$ for some matrices $P',Q' \in \bF_0^{[2] \times E(M)}$ with $\rank(Q_0'[J_1]) \le \rank(Q') \le 1$. We may assume (by applying $\bF_0$-row operations to $P_0' + \omega Q_0'$ if necessary) that the second row of $Q_0'[J_1]$ is zero. Now $R[I_1] = 0$, so we can scale each column of $\binom{P_0' + \omega Q_0'}{R}[I_1]$ to have its second entry in $\bF_0$. This yields an matrix $\binom{u_0 + \omega v_0}{R'}$ where $u_0,v_0$ are $\bF_0$-vectors, $R'$ is an $\bF_0$-matrix, and $M \dcon e \approx M\binom{u_0 + \omega v_0}{R'}$, so (\ref{b0}) holds. 
		
	\end{proof}	

\section{The Main Theorem}

	By Lemma~\ref{obad}, the abstract matroids corresponding to the represented matroids in $\cO(q)$ are not $\GF(q)$-regular. By Lemmas~\ref{lregular} and~\ref{hregular}, restrictions of $\olpg(r-1,q)$ and  $\hpg(r-1,q)$  are $\GF(q)$-regular. The following result, which applies to arbitrary $\GF(q^2)$-represented matroids, thus has Theorem~\ref{main} as a corollary. 
\begin{theorem}\label{maintech}
	Let $q$ be a prime power. If $M$ is a round rank-$r$ $\GF(q^2)$-represented matroid with a $\PG(12q^{12}+19,q)$-minor and no minor in $\cO(q)$, then $\si(M)$ is projectively equivalent to a restriction of either $M(\ha(r-1,q))$ or $M(\ba(r-1,q))$.
\end{theorem}
\begin{proof}
	Let $n = 12q^{12} + 20$ and $N$ be a $\PG(n-1,q)$-minor of $M$. Let $\cT = \cT_n(M,N)$. 
	
	 If $N$ is spanning in $M$ then, by Lemma~\ref{confinement}, we have $M \approx M(A\ |\ G_r)$ for some matrices $G_r \in \cPG(r-1,q)$ and $A$, and the result follows from Lemma~\ref{magic}. We may thus assume inductively that there exists $e \in E(M)$ so that $N$ is a minor of $M \con e$ and $\si(M \con e)$ is a restriction of either $\hpg(r-2,q)$ or $\olpg(r-2,q)$. We consider these cases in two mutually exclusive claims.
	
	\begin{claim}If the matroid $\si(M \con e)$ is projectively equivalent to a restriction of $M(\ha(r-2,q))$ then the theorem holds.\end{claim}
	\begin{proof}[Proof of claim:] The matroid $M$ is round (so is vertically $5$-connected) and has a $\GF(q^2)$-representation projectively equivalent to a submatrix of $\ha(r-2,q)$; it follows that $M$ and $e$ satisfy the hypotheses of Lemma~\ref{haugment}; Define matrices $P,Q,R$ as in the conclusion of the lemma, so $M \approx M(W)$ where $W = \binom{P + \omega Q}{R}$. 
	
	If outcome (\ref{a1}) of Lemma~\ref{haugment} holds then there is a partition $(I,J)$ of $E(M)$ so that $\rank(R[I]) + \rank(Q[J]) \le 1$, so one of these matrices is zero and the other has rank at most $1$.	If $R[I] = 0$ and $\rank(Q[J]) \le 1$ then we may perform $\GF(q)$-row-operations in the first two rows so that only the first row of $Q[J]$ is nonzero and then scale each column in $I$ so that the second entry is in $\{0,1\}$; since $R[I] = 0$ it follows that $\si(M)$ is projectively equivalent to a restriction of $M(\ha(r-1,q))$, as required.    
	
	If $Q[J] = 0$ and $\rank(R[I]) \le 1$, then let $A = W[I]$. Note that $r_M(I) \le 3$. Since $Q[J] = 0$, if the matroid $\si(M(A\ |\ G_r))$ is projectively equivalent to a restriction of $M(\ha(r-1,q))$ or $M(\ba(r-1,q))$ then so is $\si(M)$. Otherwise, $A$ is $q$-bad (recall Section~\ref{nonexamples} for a definition). By roundness of $M$ and Lemma~\ref{tanglecontractnumbers} applied with $a = h = 3$, there is a rank-$3$ minor $M'$ of $M$ with a $\PG(2,q)$-restriction $N'$  so that $E(M') = E(N') \cup I$ and $M'|I = M|I$. However $M'$ is obtained from $M$ by contracting and deleting only columns in $W[J]$, so if $G_3 \in \cPG(2,q)$ then $M' \approx M(A'\ |\ G_3)$ for some matrix $A'$ that is $\GF(q)$-row-equivalent to $A$; the matrix $A'$ is also $q$-bad, so by Lemma~\ref{magic}, the matroid $M'$ has a minor in $\cO(q)$.

	
	If outcome (\ref{a3}) of the lemma holds then let $W^+$ be the given matrix and $M^+ = M(W^+)$, noting that $M \approx M^+ \con X \del S$ and that $W^+[S \cup X]$ is strongly $q$-bad (with $Z = X$). Let $\cT^+ = \cT_n(M^+,N)$. Since $\kappa_{M^+}(S \cup X,K) \ge 4$ for each basis or cobasis $K$ of $N$, it follows that $\kappa_{\cT^+}(S \cup X) = 4$ and so, by Lemma~\ref{tanglecontractnumbers} applied with $a = 4$ and $h = 5$, $M^+$ has a minor $M'$ with a $\PG(4,q)$-restriction $N'$ so that $E(M') = E(N') \cup (S \cup X)$ and $M'|(S \cup X) = M|(S \cup X)$. Similarly to the previous case, we have $M' \approx M(B\ |\  G_5)$ for some $G_5 \in \cPG(4,q)$ and some matrix $B$ that is $\GF(q)$-row-equivalent to $W^+[S \cup X]$ and hence strongly $q$-bad. By Lemma~\ref{magic}, the matroid $M' \con X \del S$, which is a minor of $M$, has a minor in $\cO(q)$, again a contradiction.
	\end{proof}
	\begin{claim}
		If the matroid $\si(M \con e)$ is projectively equivalent to a restriction of $M(\ba(r-2,q))$ but not to a restriction of $M(\ha(r-2,q))$ then the theorem holds. 
	\end{claim}
	\begin{proof}[Proof of claim:]
		Since $M$ it is vertically $9$-connected. Since $\si(M \con e)$ is projectively equivalent to a restriction of $M(\ba(r-2,q))$, it is easy to see that $M$ and $e$ satisfy the hypotheses of Lemma~\ref{laugment}. (The required partition $(I_0,J_0)$ is induced by the line $L_0$ and its complement in the column set of $\ba(r-2,q)$.) If outcome (\ref{b0}) of the lemma holds then $\si(M \con e)$ is projectively equivalent to a restriction of $M(\ha(r-2,q))$, a contradiction. Therefore (\ref{b1}) or (\ref{b2}) holds. Let $M \approx M(W)$ where $W = \binom{P + \omega Q}{R}$ as in the lemma.
		
		Suppose that (\ref{b1}) holds, and let $(I,J)$ be the associated partition of $E(M)$. If $\si(M((W[I] \ |\ G_r)))$ is projectively equivalent to a restriction of $M(\ha(r-1,q))$ or $M(\ba(r-1,q))$ then, as $W[J]$ is a $\GF(q)$-matrix, so is $\si(M)$. Therefore we may assume that this is not the case, so $W[I]$ is $q$-bad. By roundness of $M$ we have $\kappa_{\cT}(I) = r_M(I) \le 4$, so Lemma~\ref{tanglecontractnumbers} with $a=h=4$ gives a rank-$4$ minor $M'$ of $M$ with a $\PG(3,q)$-restriction $N'$ satisfying $E(M') = E(N') \cup I$ and $M'|I = M|I$. Now $E(M) - E(M') \subseteq J$ and so $M' \approx M(B\ |\ G_4)$ for some $G_4 \in \cPG(3,q)$ and some matrix $B$ that is $\GF(q)$-row-equivalent to $W[I]$ and hence $q$-bad. Lemma~\ref{magic} implies that $M'$ has a minor in $\cO(q)$, a contradiction. 
		
		Finally, suppose that (\ref{b2}) holds. Let $W^+$ be the matrix given and let $M = M(W^+)$, noting that $M = M^+ \con X \del S$. Let $\cT^+ = \cT_n(M^+,N)$. Since $\kappa_{\cM^+}(S \cup X, K) \ge 5$ for each basis or cobasis $K$ of $N$, we have $\kappa_{\cT^+}(S \cup X) \ge 5$. By Lemma~\ref{tanglecontractnumbers} with $a = h = 6$ there is a minor $M'$ of $M^+$ and a $\PG(5,q)$-restriction $N'$ of $M'$ so that $E(M') = E(N') \cup X \cup S$, $M'|(X \cup S) = M|(X \cup S)$ and $\lambda_{M'}(X \cup S) \ge 5$, from which it follows that $6 \le r(M') \le 7$. 
		
		Since $W^+[E(M)]$ is a $\GF(q)$-matrix, we have $M' \approx M(B\ | \ G)$, where $B$ is obtained by appending a row of zeroes above $W^+[S \cup X]$ and $G$ is a $\GF(q)$-representation of $N' \cong \PG(5,q)$ with $7$ rows. (If $r(M') = 6$ then the first row of $G$ is also zero). Let $v_0, \dotsc, v_6$ denote the row vectors of $G$, so $M' \con X \del S \approx M(W')$, where
		\[W' = \left(\begin{array}{c}v_0 \\ v_1 + \omega v_4 \\ v_2 + \omega v_5 \\ v_3 + \omega v_6 \end{array}\right).\]
		
		For each $i \in \{0, \dotsc, 6\}$ let $G^i$ be the matrix obtained by removing the $i$th row of $G$. Since $\tilde{M}(G) \cong \PG(5,q)$, there is some $i \in \{0,\dotsc,6\}$ so that $\tilde{M}(G^i) \cong \PG(5,q)$. Furthermore, unless $v_0 = 0$ we may choose $i$ to be nonzero. If $v_0 = 0$ then, since $\tilde{M}(G^0) \cong \PG(5,q)$, every vector in $\GF(q^2)^4$ with first component zero is a $\GF(q)$-multiple of some column of $W'$, so $\si(M(W')) \cong \PG(2,q^2)$ and $M' \con X \del S$ clearly has a restriction in $\cO(q)$, a contradiction. 
		
		Otherwise, we can choose $i$ nonzero such that $\tilde{M}(G^i) \cong \PG(5,q)$. We will suppose that $i = 6$; the other cases are similar. Since $G^6$ contains a column from every parallel class in $\GF(q)^5$, there is some $f \in E(N')$ so that $G^6[f]$ has all entries zero except its $v_3$-entry which is nonzero. Therefore $W'[f]$ has all entries zero except its last entry which is nonzero. Now consider a representation $W''$ of $M(W') \con f$ given by removing the $f$-column and last row from $W'$. Since the matrix with rows $v_0,v_1,v_2,v_4,v_5$ has a column in every parallel class in $\GF(q)^5$, it follows that $W''$ contains a column from every parallel class in $\GF(q^2)^3$, and so $\si(M(W'')) \cong \PG(2,q^2)$ and $M(W'')$ has a restriction in $\cO(q)$, a contradiction. 
	\end{proof}
	The result now follows from the two claims. 
\end{proof}

\section*{Acknowledgements}

We thank Geoff Whittle for suggesting the problem and for his very useful discussions on the proof and techniques used.

\section*{References}

\newcounter{refs}

\begin{list}{[\arabic{refs}]}
{\usecounter{refs}\setlength{\leftmargin}{10mm}\setlength{\itemsep}{0mm}}

\item\label{d95}
J.S. Dharmatilake, 
A min-max theorem using matroid separations, 
Matroid Theory Seattle, WA, 1995, 
Contemp. Math. vol. 197, Amer. Math. Soc., Providence RI (1996), pp. 333--342

\item\label{fk91}
H. Furstenberg, Y. Katznelson, 
A density version of the Hales-Jewett Theorem, 
J. Anal. Math. 57 (1991), 64--119.

\item \label{ggrw06}
J. Geelen, B. Gerards, N. Robertson, G. Whittle, 
Obstructions to branch decomposition of matroids, 
J. Combin. Theory. Ser. B 96 (2006) 560--570.

\item\label{ggwep}
J. Geelen, B. Gerards, G. Whittle, 
Disjoint cocircuits in matroids with large rank, 
J. Combin. Theory. Ser. B 87 (2003), 270--279.

\item\label{ggwnonprime}
J. Geelen, B. Gerards, G. Whittle, 
Matroid structure. I. Confined to a subfield, 
in preparation. 

\item\label{gthesis}
B. Gerards, 
Graphs and polyhedra. Binary spaces and cutting planes, 
CWI Tract vol. 73, Stichting Mathematisch Centrum Centrum voor Wiskunde en Informatica, Amsterdam, 1990.

\item\label{n13}
P. Nelson, 
Growth rate functions of dense classes of representable matroids, 
J. Combin. Theory Ser. B 103 (2013), 75--92. 

\item \label{oxley}
J. G. Oxley, 
Matroid Theory (2nd edition),
Oxford University Press, New York, 2011.

\item\label{pvz}
R. A. Pendavingh, S.H.M. van Zwam,
Lifts of matroid representations over partial fields, 
J. Combin. Theory. Ser. B 100 (2010) 36--67.

\item\label{gmx}
N. Robertson, P. D. Seymour, 
Graph Minors. X. Obstructions to Tree-Decomposition, 
J. Combin. Theory. Ser. B 52 (1991) 153--190. 

\end{list}
\end{document}